\documentclass[a4paper,12pt]{article}

\usepackage {amsfonts, amssymb, amscd,amsthm, amsmath, eucal}

\usepackage{amsbsy}
\usepackage[all]{xy}

\usepackage{color}

\theoremstyle{plain} {
  \newtheorem{thm}{Theorem}[section]
  
  \newtheorem{cor}[thm]{Corollary}
  \newtheorem{lem}[thm]{Lemma}
  \newtheorem{prop}[thm]{Proposition}
  \theoremstyle{definition}
  \newtheorem{rem}[thm]{Remark}
    \newtheorem{constr}[thm]{Construction}
  \theoremstyle{plain}
  \newtheorem{clm}[thm]{Claim}
  \newtheorem{notation}[thm]{Notation}
  
  \newtheorem{sssection}[thm]{}
}

{

}
\renewcommand{\subsubsection}{\sssection\rm}

\newcommand{\bG}{\mathbf G}

\newcommand{\bu}{\mathbf u}
\newcommand{\bB}{\mathbf B}

\newcommand{\cB}{\mathcal B}
\newcommand{\cE}{\mathcal E}

\newcommand{\cG}{\mathcal G}

\renewcommand{\P}{\mathbb P}

\DeclareMathOperator{\Iso}{Iso}
\DeclareMathOperator{\spec}{Spec}

\newcommand{\can}{\text{\rm can}}

\newcommand{\id}{\text{\rm id}}
\newcommand{\pr}{\text{\rm pr}}

\newcommand{\inc}{\text{\rm inc}}

\newcommand{\const}{\text{\rm const}}
\newcommand{\Spec}{\text{\rm Spec}}

\newcommand{\Aut}{\text{\rm Aut}}

\newcommand{\Aff}{\mathbf {A}}
\newcommand{\Pro}{\mathbf {P}}

\newcommand \xra {\xrightarrow }

\newcommand \hra {\hookrightarrow }

\newcommand{\ttf}{{\text{f}}}

\renewcommand{\P}{\mathbb P}

\newcommand\mydim{\text{\rm dim}}

\renewcommand \id{\operatorname{id}}

\renewcommand \phi\varphi

\newcommand{\et}{\text{\rm\'et}}
\newcommand{\R}{{\rm R}}

\newcommand{\ZZ}{\mathbb Z}

\begin{document}

\title{Proof of Grothendieck--Serre conjecture
on principal bundles over regular local rings containing a finite field
}

\author{Ivan Panin\footnote{The author acknowledges support of the
RNF-grant 14-11-00456.}
}


\maketitle

\begin{abstract}
Let $R$ be a regular local ring, containing {\bf a finite field}.
Let $\bG$ be a reductive group scheme over~$R$. We prove that a principal $\bG$-bundle over~$R$ is trivial,
if it is trivial over the fraction field of $R$. In other words, if $K$ is the fraction field of $R$,
then the map of non-abelian cohomology pointed sets
\[
  H^1_{\text{\'et}}(R,\bG)\to H^1_{\text{\'et}}(K,\bG),
\]
induced by the inclusion of $R$ into $K$, has a trivial kernel.
If the regular local ring $R$ contains
{\bf an infinite field} this result is proved in \cite{FP}.
{\it Thus the conjecture holds for regular local rings containing a field.}

\end{abstract}

\section{Main results}\label{Introduction}
Let $R$ be a commutative unital ring. Recall that an $R$-group scheme $\bG$ is called \emph{reductive},
if it is affine and smooth as an $R$-scheme and if, moreover,
for each algebraically closed field $\Omega$ and for each ring homomorphism $R\to\Omega$ the scalar extension $\bG_\Omega$ is
a connected reductive algebraic group over $\Omega$. This definition of a reductive $R$-group scheme
coincides with~\cite[Exp.~XIX, Definition~2.7]{SGA3}.
A~well-known conjecture due to J.-P.~Serre and A.~Grothendieck
(see~\cite[Remarque, p.31]{Se}, \cite[Remarque 3, p.26-27]{Gr1}, and~\cite[Remarque~1.11.a]{Gr2})
asserts that given a regular local ring $R$ and its field of fractions~$K$ and given a reductive group scheme $\bG$ over $R$, the map
\[
  H^1_{\text{\'et}}(R,\bG)\to H^1_{\text{\'et}}(K,\bG),
\]
induced by the inclusion of $R$ into $K$, has a trivial kernel. The following theorem, which is the main result of the present paper,
asserts that this conjecture holds,
provided that $R$ contains {\bf a finite field}.
If $R$ contains an infinite field, then the conjecture is proved in [FP].

\begin{thm}\label{MainThm1}
Let $R$ be a regular semi-local domain containing {\bf a finite field}, and let $K$ be its field of fractions. Let $\bG$ be
a reductive group scheme over $R$. Then the map
\[
  H^1_{\text{\'et}}(R,\bG)\to H^1_{\text{\'et}}(K,\bG),
\]
\noindent
induced by the inclusion of $R$ into $K$, has a trivial kernel. In other words, under the above assumptions on $R$ and $\bG$, each principal $\bG$-bundle over $R$ having a $K$-rational point is trivial.
\end{thm}

Theorem~\ref{MainThm1} has the following
\begin{cor}
Under the hypothesis of Theorem~\ref{MainThm1}, the map
\[
  H^1_{\text{\'et}}(R,\bG)\to H^1_{\text{\'et}}(K,\bG),
\]
\noindent
induced by the inclusion of $R$ into $K$, is injective. Equivalently, if $\cG_1$ and $\cG_2$ are two principal bundles isomorphic over
$\Spec K$,
then they are isomorphic.
\end{cor}
\begin{proof}
Let $\cG_1$ and $\cG_2$ be two principal $\bG$-bundles isomorphic over
$\spec K$.
Let $\Iso(\cG_1,\cG_2)$ be the scheme of isomorphisms.
This scheme is a principal $\Aut\cG_2$-bundle.
By Theorem~\ref{MainThm1} it is trivial, and we see that $\cG_1\cong\cG_2$.
\end{proof}

Note that, while Theorem~\ref{MainThm1} was previously known for reductive group schemes $\bG$ coming from the ground field
({\it an unpublished result due to O.Gabber}), in many cases the corollary is a new result even for such group schemes.

For a scheme $U$ we denote by $\mathbb A^1_U$ the affine line over $U$ and by $\P^1_U$ the projective line over $U$.
Let $T$ be a $U$-scheme. By a principal $\bG$-bundle over $T$ we understand a principal $\bG\times_UT$-bundle.
We refer to
\cite[Exp.~XXIV, Sect.~5.3]{SGA3}
for the definitions of
a simple simply-connected group scheme over a scheme
and a semi-simple simply-connected group scheme over a scheme.

In Section~\ref{Reducing_MainThm1_to_two_others}
we deduce Theorem~\ref{MainThm1} from the following three results.


\begin{thm}{\rm{[Pan1]}}
\label{MainHomotopy}
Let $k$ be a field. Let $\mathcal
O$ be the semi-local ring of finitely many {\bf closed points} on a
$k$-smooth irreducible affine $k$-variety $X$.
Set $U=\spec \mathcal O$.
Let $\bG$ be a reductive
group scheme over $U$.
Let $\mathcal G$ be a principal $\bG$-bundle over $U$
trivial over the generic point of $U$.
Then there
exists a principal $\bG$-bundle $\mathcal G_t$ over the affine line $\mathbb A^1_U=\spec {\mathcal O}[t]$
and a monic polynomial $h(t) \in \mathcal O[t]$ such
that
\par
(i) the $\bG$-bundle $\mathcal G_t$ is trivial over the open subscheme $(\mathbb A^1_U)_h$ in $\mathbb A^1_U$ given by $h(t)\ne0$;
\par
(ii) the restriction of $\mathcal G_t$ to $\{0\}\times U$ coincides
with the original $\bG$-bundle $\mathcal G$.
\par
(iii) $h(1) \in \mathcal O$ is a unit.

\end{thm}
If the field $k$ is infinite a weaker result is proved in
\cite[Thm.1.2]{PSV}.

\begin{thm}
\label{th:psv}
Let $k$ be a field. Let $\mathcal O$ be the semi-local ring of finitely many {\bf closed points} on a
$k$-smooth irreducible affine $k$-variety $X$.
Set $U=\spec \mathcal O$.
Let $\bG$ be {\bf a simple simply-connected} group scheme over $U$.
Let $\cE_t$ be a principal $\bG$-bundle over the affine line $\mathbb A^1_U=\spec {\mathcal O}[t]$, and let $h(t)\in \mathcal O[t]$ be a monic polynomial.
Denote by $(\mathbb A^1_U)_h$ the open subscheme in $\mathbb A^1_U$ given by $h(t)\ne0$ and assume that
the restriction of $\cE_t$ to $(\mathbb A^1_U)_h$ is a trivial principal $\bG$-bundle.
Then for each section $s:U\to\mathbb A^1_U$ of the projection $\mathbb A^1_U\to U$ the $\bG$-bundle $s^*\cE_t$ over $U$ is trivial.
\end{thm}
If the field $k$ is infinite this result is proved in
\cite[Thm.2]{FP}.

\begin{thm}{\rm{[Pan2]}}
\label{MainThmGeometric}
Let $k$ be a field. Assume that for any irreducible $k$-smooth affine variety $X$ and any
finite family of its closed points $x_1, x_2,\dots, x_n$ and the semi-local $k$-algebra
$\mathcal O:= \mathcal O_{X,x_1, x_2,\dots, x_n}$
and all semi-simple simply connected reductive $\mathcal O$-group schemes $H$
the pointed set map
$$H^1_{\text{\rm\'et}}(\mathcal O, H) \to H^1_{\text{\rm\'et}}(k(X), H),$$
\noindent
induced by the inclusion of $\mathcal O$ into its fraction field $k(X)$, has trivial kernel.

Then for any regular semi-local domain $\mathcal O$ of the form
$\mathcal O_{X,x_1, x_2,\dots, x_n}$ above
and any reductive $\mathcal O$-group scheme $G$
the pointed set map
$$H^1_{\text{\rm\'et}}(\mathcal O, G) \to H^1_{\text{\rm\'et}}(k(X), G),$$
\noindent
induced by the inclusion of $\mathcal O$ into its fraction field $k(X)$, has trivial kernel.
\end{thm}


Theorem \ref{th:psv} is an easy consequence of
Theorem
\ref{MainThm2}
proven in section
\ref{Section_Theorem_MainThm2}
and of
Proposition \ref{SchemeY}
proven in section \ref{SchemeY_section}.
To state Theorem \ref{MainThm2} recall a notion.
Let $V$ be a semi-local scheme. We will call a simple $V$-group scheme
{\it quasi-split} if its restriction to each connected component of $V$ contains {\bf a Borel subgroup scheme}.

\begin{thm}
\label{MainThm2}
Let $k$ be a {\bf finite} field.
Let
$U$ and $\bG$ be as in Theorem~\ref{th:psv}.
Let $Z\subset\mathbb P^1_U$ be a closed subscheme finite over $U$.
Let $Y\subset\mathbb P^1_U$ be a closed subscheme finite and
\'etale over $U$ and such that\\
(i) $\bG_Y:=\bG\times_U Y$ is quasi-split, \\
(ii) $Y\cap Z=\emptyset$ and $Y \cap \{\infty\}\times U=\emptyset= Z \cap \{\infty\}\times U$, \\
(iii) for any closed point $u \in U$ one has $Pic(\mathbb P^1_u - Y_u)=0$, where $Y_u:=\mathbb P^1_u\cap Y$.\\
Let $\mathcal G$ be a~principal $\bG$-bundle over
$\mathbb P^1_U$ such that its restriction to
$\mathbb P^1_U- Z$ is trivial.
Then the restriction of $\mathcal G$ to
$\mathbb P^1_U-Y$ is also trivial.\\
In particular, the principal $\bG$-bundle $\mathcal G$ is trivial locally for the Zarisky topology.
\end{thm}
If the field $k$ is infinite then a stronger result is proved in
\cite[Thm.3]{FP}.

The article is organized as follows.
Theorem
\ref{MainThm2}
proven in section
\ref{Section_Theorem_MainThm2}.
In Section a useful lemma is proved, namely Lemma \ref{F1F2}.
In Section \ref{SchemeY_section} Theorem~\ref{th:psv} is proved
using the theorem \ref{MainThm2}.
Theorem \ref{MainThm1} is proved in Section \ref{Reducing_MainThm1_to_two_others}
as a consequence of Theorems~\ref{MainHomotopy}, \ref{th:psv}, \ref{MainThmGeometric}.

Describe informally the method of the proof of Theorem \ref{MainThm1}
for the geometric case and even for
the case when the ring $R$ is the semi-local ring of finitely many closed points
on an affine $k$-smooth variety $X$. We may and will assume that the group scheme
$\bf G$ is a reductive group scheme over $X$. We may and will assume that
the principal $\bf G$-bundle $\mathcal G$ is defined over $X$ and is trivial over $X-Z$
for a non-empty closed subset $Z$ of $X$. A trivialization of $\mathcal G$ over $X-Z$
"defines" a $\bf G$-bundle $\tilde {\mathcal G}$ over the motivic space $X/(X-Z)$
such that $p^*(\tilde {\mathcal G})=\mathcal G$ for the projection $X\to X/(X-Z)$.
Recall that the motivic space $X/(X-Z)$ is just the Nisnevich sheaf $X/(X-Z)$.

Let $U=Spec(R)$ and $can: U\to X$ be the canonical morphism.
We construct a closed subset $Y \subset \mathbb A^1_U$ finite and \'{e}tale over $U$ and
a morphism of Nisnevich sheaves (= motivic spaces)
$\theta: \mathbb A^1_U - Y \to X/(X-Z)$
such that \\
(1) $\theta^*(\tilde {\mathcal G})$ is the trivial $\bf G$-bundle,\\
(2) $1\times U \subset \mathbb A^1_U - Y$, \\
(3) the morphism $\theta|_{0\times U}$ coincides with the one
$U\xrightarrow{can} X \xrightarrow{p} X/(X-Z)$.\\
Given that morphism $\theta$ one has a chain of equalities
$$\mathcal G|_U=can^*(\mathcal G)=can^*(p^*(\tilde {\mathcal G}))=\theta^*(\tilde {\mathcal G})|_{1\times U}.$$
The properties (1) and (2) yield the triviality of $\mathcal G|_U$. {\it This proves Theorem
\ref{MainThm1} for the case of simply-connected $\bf G$.
The simply-connectness of ${\bf G}$ is used to find the desired closed subset $Y$ in $\mathbb A^1_U$.
}
Two purity theorems
\cite[Thm.1.1, Thm.1.2]{Pan2}
are used to reduce the case of general reductive ${\bf G}$
to the case of simply-connected $\bf G$.
Ones the proof of Theorem \ref{MainThm1} is completed in Section \ref{Reducing_MainThm1_to_two_others}
we derive from Theorem \ref{MainThm1} two short exact sequences
(see Theorems \ref{Aus_Buksbaum_3} and \ref{PurityForSubgroup}).

We work over a finite base field $k$ in the present paper.
Hence we use systematically results of \cite{Poo1} and \cite{Poo2}.
Nevetheless, there are other vast difficulties.
Namelly, often
for a closed subscheme $S$ of an affine $k$-smooth curve $C$ there is no closed embedding
$S$ into $\mathbb A^1_k$.
Indeed, such a finite scheme $S$ can be embedded in $\mathbb A^1_k$ if and only if
for any finite field extension $l/k$ the amount of the $l$-rational points of $S$
is at most the amount of $l$-rational points in $\mathbb A^1_k$.
To overcome those difficulties we replace systematically regular morphisms
with appropriate variant of finite correspondences
(see for instance Lemma \ref{F1F2}.


A good reference for the history of the topic is
\cite{FP}.

The author thanks A.~Suslin for his interest in the topic of the present article. He
also thanks to A.Stavrova for paying my attention to Poonen's works on Bertini type theorems
for varieties over finite fields. He thanks D.Orlov for useful comments concerning
the weighted projective spaces tacitely involved in the construction of elementary fibrations.
He thanks M.Ojanguren for many inspiring ideas arising from our joint works with him.

\section{Proof of Theorem~\ref{MainThm2}}
\label{Section_Theorem_MainThm2}
\begin{proof}
Let $Y'$ be a semi-local scheme. Recall that a simple $Y'$-group scheme
{\it quasi-split} if its restriction to each connected component of $Y'$ contains {\bf a Borel subgroup scheme}.
Let $Y$ be a semi-local scheme. Recall that a simple $Y$-group scheme {\it isotropic},
if its restriction to each connected component of $Y$ contains a proper parabolic subgroup scheme.
Any quasi-split $Y'$-group scheme is isotropic.

Clearly, the closed subscheme $Y\subset \Pro^1_U$ satisfies all the hypotheses
of the Theorem 3 from \cite{FP}, possibly, except of one. Namely,
it is not necessary that for any closed point $u\in U$ the scheme
$Y_u \subset \Pro^1_u$
contains a $k(u)$ rational point.
However, $Y$ satisfies the following property:

\begin{itemize}
\item
For any closed point $u\in U$ let
$\mathcal G_u:= \mathcal G\times_U u$
be the principal $\bG_u$-bundle, where
$\bG_u=\bG\times_U u$ is the $k(u)$-group scheme.
Then $\mathcal G_u|_{\Pro^1_u-Y_u}$ is trivial.
\end{itemize}
This is exactly the property of $Y$, which is used to
prove the lemma 5.21 from
\cite{FP}.
With this property of $Y$ in hand the proof
of Theorem \ref{MainThm2} {\it literally repeats}
the proof of \cite[Thm.3]{FP}.

So, to complete the proof of Theorem \ref{MainThm2} it remains
to check the latter property of $Y$.
The $\bG_u$-bundle $\mathcal G_u$ is trivial over $\mathbb A^1_u-Z_u$. Thus,
by~\cite[Corollary~3.10(a)]{GilleTorseurs}
it is trivial locally for Zariski topology on $\mathbb P^1_u$.
Using again
\cite[Corollary~3.10(a)]{GilleTorseurs},
we see that $\mathcal G_u$ is trivial over $\P^1_u-Y_u$.
The proof of Theorem \ref{MainThm2} is completed.
\end{proof}

\begin{rem}\label{comment}
We recommend to the reader take a look over arguments given
in the outline of the proof of \cite[Theorem 3]{FP}
(see \cite[Subsection 5.1]{FP}).
\end{rem}

\section{A useful lemma}\label{Appendix}
We will need the following lemma
\begin{lem}
\label{F1F2}
Let $U$ be as in the Proposition \ref{SchemeY}.
Let $Z\subset\mathbb A^1_U$ be a closed subscheme finite over $U$.
Let $Y^{\prime} \to U$ be a finite \'{e}tale morphism such that
for any closed point $u_i$ in $U$ the fibre $Y^{\prime}_{u_i}$ of $Y^{\prime}$ over $u_i$
contains a $k(u_i)$-rational point. Then there are finite field extensions
$k_1$ and $k_2$ of the finite field $k$ such that \\
(i) the degrees $[k_1: k]$ and $[k_2: k]$ are coprime,\\
(ii) $k(u_i) \otimes_k k_r$ is a field for $r=1$ and $r=2$,\\
(iii) the degrees $[k_1: k]$ and $[k_2: k]$ are strictly greater than any of the degrees
$[k(z): k(u)]$, where $z$ runs over all closed points of $Z$,\\
(iv) there is a closed embedding of $U$-schemes
$Y^{\prime\prime}=((Y^{\prime}\otimes_k k_1) \coprod (Y^{\prime}\otimes_k k_2)) \xrightarrow{i} \mathbb A^1_U$,\\
(v) for $Y=i(Y^{\prime\prime})$ one has $Y \cap Z = \emptyset$,\\
(vi) for any closed point $u_i$ in $U$ one has
$Pic(\P^1_{u_i}-Y_{u_i})=0$.
\end{lem}
Let $k$ be a {\bf finite} field. Let $\mathcal O$ be the semi-local ring of finitely many {\bf closed points} on a
$k$-smooth irreducible affine $k$-variety $X$.
Set $U=\spec \mathcal O$. Let ${\bf u}\subset U$ be the set of all closed points in $U$.
For a point $u\in {\bf u}$ let $k(u)$ be its residue field.

\begin{notation}\label{notn: F1F2}
Let $k$ be the finite field of characteristic $p$,
$k'/k$ be a finite field extensions.
Let $c=\sharp(k)$ (the cardinality of $k$).
For a positive integer $r$ let $k'(r)$
be a unique field extension of the degree $r$ of the field $k'$.
Let $\mathbb A^1_{k}(r)$ be the set of all degree $r$ points
on the affine line $\mathbb A^1_{k}$.
Let $Irr(r)$ be the number of the degree $r$ points on
$\mathbb A^1_{k}$.

\end{notation}

\begin{lem}\label{l:F1F2_preliminary}
Let $k$ be the finite field of characteristic $p$,
$c=\sharp(k)$,
$k^{\prime}/k$ be a finite field extension of degree $d$.
Let $q\in \mathbb N$
be {\bf a prime} which is {\bf co-prime} as to the characteristic $p$ of the field $k$,
so to the integer $d$.
Then \\
(1)$Irr(q)=(c^q-c)/q$;\\
(2) the $k$-algebra $k^{\prime}\otimes_{k} k(q)$ is the field $k^{\prime}(q)$;\\
(3) $Irr(dq) \geq Irr(q)$;\\
(4) $Irr(dq)\geq  (c^q-c)/q \gg 0$ for $q\gg 0$.
\end{lem}

\begin{proof}
The assertions (1) and (2) are clear. The assertion (3) it is equivalent
to the unequality $\varphi(c^{dq}-1)/dq \geq \varphi(c^q-1)/q$. The latter is equivalent
to the one $\varphi(c^{dq}-1)/d \geq \varphi(c^q-1)$.

The norm map $N: k(dq)^{\times} \to k(q)^{\times}$ is surjective.
Both groups are the cyclic groups of orders $c^{dq}-1$ and $c^q-1$ respectively.
Thus for any $c^q-1$-th primitive root of unity $\xi \in k(q)^{\times}$
there is a $c^{dq}-1$-th primitive root of unity $\zeta \in k(dq)^{\times}$
such that $Norm(\zeta)=\xi$. On the other hand, if $N(\zeta)=\xi$
and $\sigma \in Gal(k(dq)/k(q))$, then $N(\zeta^{\sigma})=\xi$.
Thus $\varphi(c^{dq}-1)/d \geq \varphi(c^q-1)$. The assertion (3) is proved.

The assertions (1) and (3) yield the assertion (4).
\end{proof}

\begin{notation}
\label{d(Y_u)}
For any \'{e}tale $k$-scheme $W$
set $d(W)=\text{max} \{ deq_{k} k(v)| v\in W \}$.
\end{notation}

\begin{lem}\label{l:F1F2_surjectivity}
Let $Y_u$ be an \'{e}tale $k$-scheme. For any positive integer $d$ let
$Y_u(d)\subseteq Y_u$ be the subset consisting of points $v\in Y_u$ such that
$deg_{k}(k(v))=d$.
For any prime $q\gg 0$ the following holds: \\
(1) if $v\in Y_u$, then $k(v)\otimes_{k} k(q)$ is the field $k(v)(q)$ of the degree $q$ over $k(v)$; \\
(2) $k[Y_u]\otimes_{k} k(q)=(\prod_{v\in Y_u}k(v))\otimes_{k} k(q) =\prod_{v\in Y_u} k(v)(q)$;\\
(3) there is a surjective
$k$-algebra homomorphism
\begin{equation}\label{eq:first_surjectity}
\alpha: k[t]\to k[Y_u]\otimes_{k} k(q)
\end{equation}
\end{lem}

\begin{proof}
The first assertion follows from the lemma \ref{l:F1F2_preliminary}(2).
The second assertion follows from the first one.
Prove now the third assertion.
The $k$-algebra homorphism
$$k[t]\to \prod^{d(Y_u)}_{d=1}(\prod_{x\in \mathbb A^1_{k}(dq))}k[t]/\mathfrak(m_x))=
\prod^{d(Y_u)}_{d=1}(\prod_{x\in \mathbb A^1_{k}(qd)} k(dq))$$
is surjective. The $k$-algebra $k[Y_u]$ is equal to
$\prod^{d(Y_u)}_{d=1}\prod_{v\in Y_u(d)}k(v)=\prod^{d(Y_u)}_{d=1}\prod_{v\in Y_u(d)}k(d)$.
Thus for any prime $q\gg 0$ one has
$$k[Y_u]\otimes_{k} k(q)=\prod^{d(Y_u)}_{d=1}\prod_{v\in Y_u(d)}k(dq).$$
Choose a prime $q\gg 0$ such that for any $d=1,2,...,d(Y_u)$ one has
$Irr(dq)\geq \sharp(Y_u(d))$.
This is possible by the lemma \ref{l:F1F2_preliminary}.
In this case for any $d=1,2,...,d(Y_u)$ there exists
a surjective $k$-algebra homomorphism
$\prod_{x\in \mathbb A^1_{k}(qd)} k(dq)\to \prod^{d(Y_u)}_{d=1}\prod_{v\in Y_u(d)}k(dq)$.
Thus for this specific choice of prime $q$ there exists a surjective $k$-algebra homomorphism
$k[t]\to k[Y_u]\otimes_{k} k(q)$.
The third assertion of the lemma is proved.
\end{proof}

\begin{lem}
\label{F1F2_copy_very_final}
Let $k$ be the finite field from the lemma \ref{F1F2}.
Let $U$ be the the semi-local scheme as in the lemma \ref{F1F2}.
Let $u \in U$ be a closed point and let $k(u)$ be its residue field (it is a finite extension of the finite field $k$).
Let $Z_u\subset\mathbb A^1_u$ be a closed subscheme finite over $u$.
Let $Y^{\prime}_u \to u$ be a finite \'{e}tale morphism such that
$Y^{\prime}_u$
contains a $k(u)$-rational point. Then for any
two different primes $q_1,q_2\gg 0$
the finite field extensions
$k(q_1)$ and $k(q_2)$ of the finite field $k$ satisfy the following conditions \\
(i) for $j=1,2$ one has $q_j > d(Y'_u)$;\\
(ii) for any $j=1,2$ and any point $v\in Y'_u$ the $k$-algebra $k(v) \otimes_k k(q_j)$ is a field;\\
(iii) both primes $q_1,q_2$ are strictly greater than
$\text{max}\{[k(z): k]| z\in Z_u \}$;\\
(iv) there is a closed embedding of $u$-schemes
$Y^{\prime\prime}_u=((Y^{\prime}_u\otimes_k k_1) \coprod (Y^{\prime}_u\otimes_k k_2)) \xrightarrow{i_u} \mathbb A^1_u$,\\
(v) for $Y_u=i_u(Y^{\prime\prime}_u)$ one has $Y_u \cap Z_u = \emptyset$,\\
(vi) $Pic(\P^1_{u}-Y_u)=0$.\\
\end{lem}

\begin{proof}

Clearly the conditions (i) and (iii) holds for all different big enough primes $q_1,q_2$.
Applying now the lemma \ref{l:F1F2_surjectivity}(1) to the \'{e}tale $k$-scheme $Y'_u$
we see that
the condition (ii)
holds also for all different big enough primes $q_1,q_2$.
Take now any two different primes $q_1,q_2$
satisfying the conditions (i) to (iii).
{\it We claim that those two primes are such that the conditions
(iv) to (vi) are satisfied.
}

Check the condition (iv).
Lemma \ref{l:F1F2_surjectivity}(3)
shows that there are surjections of the $k$-algebras
$\alpha_j: k[t] \to k[Y'_u]\otimes_k k_i$ for $j=1,2$.
Using the structure morphism
$Y_u\to u$ extend $\alpha_j$'s
to surjections of the $k(u)$-algebras
$\beta_j: k(u)[t] \to k[Y'_u]\otimes_k k_j$ for $j=1,2$.
The surjections $\beta_j$'s induce two closed embeddings
$i_j: Y^{\prime}_u\otimes_k k_j \hookrightarrow \mathbb A^1_{k(u)}$
of $k(u)$-schemes.
By the lemma \ref{l:F1F2_surjectivity}(1)
for any point $w\in Y^{\prime}_u\otimes_k k_j$
one has
$deg_k(w)=d\cdot q_j$
for appropriate $d=1,2,...,d(Y'_u)$.
Since the primes $q_1$, $q_2$ do not coincide and
$q_j > d(Y'_u)$ for $j=1,2$, hence for any
$d_1,d_2 \in \{1,2,...,d(Y'_u)\}$
one has
$d_1q_1\neq d_2q_2$.
Thus the intersection of
$i_1(Y^{\prime}_u\otimes_k k_1)$ and $i_2(Y^{\prime}_u\otimes_k k_2)$ is empty.
Hence the $k(u)$-morphism
$$i_u=i_1\sqcup i_2: Y^{\prime}_u\otimes_k k_1 \sqcup Y^{\prime}_u\otimes_k k_2 \to \mathbb A^1_u$$
is a closed embedding. The condition (iv) is verified.
Since for any point $w\in Y^{\prime}_u\otimes_k k_j$
one has
$deg_k(w)=d\cdot q_j\ge q_j$,
hence the above choice of $q_1$ and $q_2$
yield
the condition (v).

Set now
$Y_u=i_u(Y^{\prime\prime}_u)$, where $Y^{\prime\prime}_u=Y^{\prime}_u\otimes_k k_1 \sqcup Y^{\prime}_u\otimes_k k_2$.
Since $Y^{\prime}_u$
contains a $k(u)$-rational point,
hence
$Y_u$ has a point of the degree $q_1$ over $k(u)$ and a point of the degree $q_2$ over $k(u0$.
Thus $Pic(\mathbb P^1_u-Y_u)=0$. This proves the condition (vi).
\end{proof}


\begin{proof}[Proof of the lemma \ref{F1F2}]
We prove that lemma for the case of local $U$ and left the general case to the reader.
So we may suppose that there is only one closed point $u$ in $U$.
Let $k(u)$ be its residue field.
Let $Y^{\prime} \to U$ be the finite \'{e}tale morphism from the lemma \ref{F1F2}.
Let $Y^{\prime}_{u}=Y'\times_U u$ be the fibre of $Y^{\prime}$ over $u$.
By the hypotheses of the lemma \ref{F1F2}
the $k(u)$-scheme $Y^{\prime}_{u}$ contains a $k(u)$-rational point.
Let $k_1/k$ and $k_2/k$ be the field extensions from
the lemma \ref{F1F2_copy_very_final}.
Set
$Y''=(Y^{\prime}\otimes_k k_1) \sqcup (Y^{\prime}\otimes_k k_2)$
and
$Y''_u=Y'\times_U u$.
Then one has
$Y''_u=(Y^{\prime}_u\otimes_k k_1) \sqcup (Y^{\prime}_u\otimes_k k_2)$.

Let
$i_u: Y''_u \hookrightarrow \mathbb A^1_u$
be the closed embedding from the lemma
\ref{F1F2_copy_very_final}. Then the pull-back map of the $k$-algebras
$i^*_u: k(u)[t]\to k[Y''_u]$ is surjective.
Set $\alpha=i^*_u(t)\in k[Y''_u]$.
Let $A\in k[Y'']$ be a lift of $\alpha$.

Consider a unique $k[U]$-algebra homomorphism
$k[U][t]\xrightarrow{i^*} k[Y'']$
which takes $t$ to $A$.
Since $k[Y'']$ is a finitely generated $k[U]$-module,
the Nakayama lemma shows that the map $i^*$ is surjective.
Hence the induced scheme morphism
$i: Y''\to \mathbb A^1_U$
is a closed embedding.
The assertion (iv) proved.
Set $Y=i(Y'')$.

Clearly, the two closed embeddings
$Y''_u\to Y''\to \mathbb A^1_U$
and
$Y''_u\to \mathbb A^1_u \to \mathbb A^1_U$
coincide.
Lemma \ref{F1F2_copy_very_final} completes now the proof of the lemma
\ref{F1F2}.
\end{proof}

\section{Proof of Theorem \ref{th:psv}.}
\label{SchemeY_section}
Theorem \ref{th:psv} in the case of an infinite field $k$ is proved in \cite[Thm. 2]{FP}.
{\it So, we prove the remaining case, when the field $k$ is finite.
}

Let $k$, $U$ and $\bG$ be as in Theorem~\ref{th:psv} and
{\it let the field $k$ be finite.
}
Let $u_1,\ldots,u_n$ be all the closed points of $U$. Let $k(u_i)$ be the residue field of $u_i$. Consider the reduced closed subscheme $\bu$ of $U$,
whose points are $u_1$, \ldots, $u_n$. Thus
\[
 \bu\cong\coprod_i\spec k(u_i).
\]
Set $\bG_\bu=\bG\times_U\bu$. By $\bG_{u_i}$ we denote the fiber of $\bG$ over $u_i$;
it is a simple simply-connected algebraic group over $k(u_i)$.
The following lemma is a simplified version of Lemma
\cite[Lemma 3.3]{Pan1}.

\begin{lem}
\label{OjPan_1}
Let $S=\text{Spec}(A)$ be a regular
semi-local scheme such that
{\bf the residue field at any of its closed point is finite}.
Let $T$ be a closed
subscheme of $S$. Let $W$ be a closed subscheme of
the projective space $\Bbb P^d_S$.
Assume that over $T$ there exists a section
$\delta:T\to W$
of the projection $W\to S$. Suppose further that
$W$ is $S$-smooth and equidimensional over $S$ of relative dimension $r$.
Then there exists a closed subscheme $\tilde S$ of $W$ which is finite
\'etale over
$S$ and contains $\delta(T)$.
\end{lem}
\begin{proof}
If $r=0$, then there is nothing to prove.
So, we will assume below in the proof that $r\geq 1$.
Since $S$ is semilocal, after a linear change
of coordinates we may assume that $\delta$ maps $T$ into
the closed subscheme of $\Pro^d_T$ defined by
$X_1=\dots=X_d=0$.

{\bf Let $s\in S$ be a closed point of $S$. Let $s\in T$ be a closed point of $T$.
Let $\Pro^d_s$ be closed fibre  of $\Pro^d_S$ over the point $s$.
Let $\mathbb A^d_s\subset \Pro^d_s$ be the affine subspace defined by the unequality
$X_0\neq 0$.
Let $t_i:=X_i/X_0$ ($i \in \{1,2, \dots, d\}$) be the coordinate function on $\mathbb A^d_s$.
The origin $x_{0,s}=(0,0,\dots, 0)\in \mathbb A^d_s$ has the homogeneous coordinates
$[1:0:\dots :0]$ in $\Pro^d_s$.
Let $W(s)\subset \mathbb P^d_s$ be the fibre of $W$ over the point $s$.
If $x_{0,s}$ is in $W(s)$, then
let $\tau_s\subset \mathbb A^d_s$ be the tangent space to $W(s)$ at the point
$x_0\in \mathbb A^d_s$.
In this case let $l_s=l_s(t_1,t_2,\dots, t_d)$ be a linear form in $k(s)[t_1,t_2,\cdots,t_d]$
such that
$l_s|_{\tau_s}\neq 0$.
If $x_{0,s}$ is not in $W(s)$, then set $l_s=t_1$.
In all the cases
$$L_s:=X_0\cdot l_s \in k(s)[X_1,X_2,\cdots,X_d]$$
is a homogeneous polinomial of degree $1$.

Let $s \in S$ be a closed point. Suppose $x_{0,s} \in W(s)$.
Then
by \cite[Thm. 1.2]{Poo1} there is an integer $N_1(s)\geq 1$ such that for any
positive integer $N\geq N_1(s)$ there is a homogeneous polinomial
$$H_{1,N}(s)=X^N_0\cdot F_{0,N}(s)+X^{N-1}_0\cdot F_{1,N}(s)+\dots+X_0\cdot F_{N-1,N}(s)+F_{N,N}(s)\in k(s)[X_0,X_1,\cdots,X_d]$$
of degree $N$ with homogeneous polinomials $F_{i,N}(s)\in k(s)[X_1,\dots,X_n]$ of degree $i$
such that the following holds \\
(i) $F_{0,N}(s)=0$, $F_{1,N}(s)=L_s$; \\
(ii) the subscheme $V(s)\subseteq \Pro^d_s$ defined by the equation $H_{1,N}(s)=0$
intersects $W(s)$ transversally.

Let $s \in S$ be a closed point. Suppose $x_{0,s}$ is not in $W(s)$.
Then
by \cite[Thm. 1.2]{Poo1} there is an integer $N_1(s)\geq 1$ such that for any
positive integer $N\geq N_1(s)$ there is a homogeneous polinomial
$$H_{1,N}(s)=X^N_0\cdot F_{0,N}(s)+X^{N-1}_0\cdot F_{1,N}(s)+\dots+X_0\cdot F_{N-1,N}(s)+F_{N,N}(s)\in k(s)[X_0,X_1,\cdots,X_d]$$
of degree $N$ with homogeneous polinomials $F_{i,N}(s)\in k(s)[X_1,\dots,X_n]$ of degree $i$
such that the following holds \\
(i) $F_{0,N}(s)=0$, $F_{1,N}(s)=L_s=X_1$; \\
(ii) the subscheme $V(s)\subseteq \Pro^d_s$ defined by the equation $H_{1,N}(s)=0$
intersects $W(s)$ transversally.

Let $N_1=\text{max}_s \{N_1(s)\}$, where $s$ runs over all the closed points of $S$.
For any closed point $s \in S$ set $H_1(s):= H_{1,N_1}(s)$.
Then for any closed point $s \in S$ the polinomial $H_1(s)\in k(s)[X_0,X_1,\cdots,X_d]$ is homogeneous of the degree $N_1$.
By the chinese
remainders' theorem for any $i=0,1,\dots,N$
there exists a common  lift
$F_{i,N_1}\in A[X_1,\dots,X_d]$ of all polynomials $F_{i,N_1}(s)$,
$s$ is a closed point of $S$,
such that
$F_{0,N_1}=0$ and for any $i=0,1,\dots,N$ the polinomial $F_{i,N_1}$ is homogeneous of
degree $i$.
By the chinese
remainders' theorem
there exists a common lift
$L \in A[X_1,\cdots,X_d]$
of all polynomials $L_s$,
$s$ is a closed point of $S$,
such that
$L$ is homogeneous of the degree one.
Set
$$H_{1,N_1}:=X^{N_1-1}_0\cdot L+X^{N_1-2}_0\cdot F_{2,N_1} \dots+X_0\cdot F_{N_1-1,N_1}+F_{N_1,N_1}\in A[X_0,X_1,\cdots,X_d].$$
Note that for any closed point $s$ in $S$ the evaluation of $H_{1,N_1}$ at the point $s$ coincides with the polinomial
$H_{1,N_1}(s)$ in $k(s)[X_0,X_1,\cdots,X_d]$.
Note also that
$H_{1,N_1}[1:0:\dots:0]=0$.
Hence $H_{1,N_1}|_{\delta(T)}\equiv 0$.

Repeating this consruction several times
we can choose a family of $r$ {\bf homogeneous} polynomials $H_{1,N_1},\dots,H_{r,N_r}$
(in general of increasing degrees)
such that for any closed point $s$ in $S$
a subscheme in $\Pro^d_s$ defined by the equations
$H_{1,N_1}(s)=0\,,\,\dots\,,\, H_{r,N_r}(s)=0$
intersects $W(s)$ transversally.
Let $Y$ be a subscheme in $\Pro^d_S$ defined by the equations
$$H_{1,N_1}=0\,,\,\dots\,,\, H_{r,N_r}=0.$$

For a closed point $s$ in $S$ let $Y(s)$ be the fibre of $Y$ over $s$.
It coincides with the subscheme in $\Pro^d_s$ defined by the equations
$H_{1,N_1}=0\,,\,\dots\,,\, H_{r,N_r}=0$. Thus, for any closed point
$s$ in $S$ the scheme $Y(s)$ intersect $W$ transversally.
}

{\it We claim that the subscheme  $\tilde S=W\cap Y$  has the
required properties}. Note first that $W\cap Y$ is finite
over $S$.
In fact, $W\cap Y$
is projective over $S$ and every closed
fibre (hence every fibre) is finite. Since the closed
fibres of $W\cap Y$ are finite \'etale over the closed points
of $S$, to show that $W\cap Y$ is finite \'etale over $S$ it
only remains to show that it is flat over $S$. Noting that
$W\cap Y$ is defined in every closed fibre by a regular
sequence of equations and localizing at each closed point of
$S$, we see that flatness follows from
\cite[Lemma 7.3]{OP2}.

\end{proof}

\begin{prop}
\label{SchemeY}
Let $Z\subset\mathbb A^1_U$ be a closed subscheme finite over $U$.
There is a closed subscheme $Y\subset\mathbb A^1_U$ which is \'etale and finite over $U$ and such that \\
(i) $\bG_Y:=\bG\times_UY$ is quasi-split, \\
(ii) $Y\cap Z=\emptyset$,\\
(iii) for any closed point $u \in U$ one has $Pic(\P^1_u - Y_u)=0$, where $Y_u:=\P^1_u\cap Y$.\\
(Note that $Y$ and $Z$ are closed in $\P^1_U$ since they are finite over $U$).
\end{prop}

\begin{proof}
For every $u_i$ in $\bu$ choose a Borel subgroup $\bB_{u_i}$ in $\bG_{u_i}$.
{\it The latter is possible since the fields $k(u_i)$ are finite.}
Let $\cB$ be the $U$-scheme of Borel subgroup schemes of $\bG$.
It is a smooth projective $U$-scheme (see~\cite[Cor.~3.5, Exp.~XXVI]{SGA3}).
The subgroup $\bB_{u_i}$ in $\bG_{u_i}$ is a $k(u_i)$-rational point~$b_i$ in the fibre of $\cB$ over the point $u_i$.
Now apply Lemma
\ref{OjPan_1} to
the scheme
$U$ for $S$,
the scheme $\mathbf u$ for $T$,
the scheme $\cB$ for $W$
and to a section $\delta: \mathbf u \to \cB$,
which takes the point $u_i$ to the point $b_i \in \cB$.
Since the scheme $\cB$ is $U$-smooth and
is equidimensional over $S$
we are under the assumption of Lemma \ref{OjPan_1}. Hence
there is a closed subscheme $Y^{\prime}$ of $\cB$ such that
$Y^{\prime}$ is \'etale over $U$ and all the $b_i$'s are in $Y$
The $U$-scheme $Y^{\prime}$ satisfies the hypotheses of Lemma
\ref{F1F2}. Take the closed subscheme $Y$ of $\mathbb A^1_U$ as in the item (v)
of the Lemma. For that specific $Y$ the conditions (ii) and (iii) of the Proposition are
obviously satisfied. The condition (i) is satisfied too, since already
it is satisfied for the $U$-scheme $Y^{\prime}$. The Proposition follows.

\end{proof}

\begin{proof}[Proof of Theorem~\ref{th:psv}]
Set $Z:=\{h=0\}\cup s(U)\subset\mathbb A^1_U$.
Clearly, $Z$ is finite over $U$.
Since the principal $\bG$-bundle $\cE_t$ is trivial over $(\mathbb A^1_U)_h$
it is trivial over $\mathbb A^1_U-Z$.
Note that $\{h=0\}$ is closed in $\P^1_U$ and finite over $U$ because $h$ is monic.
Further, $s(U)$ is also closed in $\P^1_U$ and finite over $U$ because it is a zero set of a degree one monic polynomial.
Thus $Z\subset\P^1_U$ is closed and finite over $U$.


Since the principal $\bG$-bundle $\cE_t$ is trivial over $(\mathbb A^1_U)_h$, and
$\bG$-bundles can be glued in Zariski topology,
there exists a principal $\bG$-bundle $\cG$ over $\P^1_U$ such that\\
\indent (i) its restriction to $\mathbb A^1_U$ coincides with $\cE_t$;\\
\indent (ii) its restriction to $\P^1_U-Z$ is trivial.

Now choose $Y$ in $\mathbb A^1_U$ as in Proposition
~\ref{SchemeY}.
Clearly, $Y$ is finite \'{e}tale over $U$ and closed in $\P^1_U$.
Moreover,
$Y \cap \{\infty\}\times U=\emptyset= Z \cap \{\infty\}\times U$
and
$Y\cap Z= \emptyset$.
Applying
Theorem~\ref{MainThm2}
with this choice of $Y$ and $Z$,
we see that the restriction of $\cG$ to $\P^1_U-Y$ is a trivial $\bG$-bundle.
Since $s(U)$ is in $\mathbb A^1_U-Y$
and $\cG|_{\mathbb A^1_U}$ coincides with $\cE_t$,
we conclude that $s^*\cE_t$ is a trivial principal $\bG$-bundle over $U$.
\end{proof}

\section{Proof of Theorem~\ref{MainThm1}}
\label{Reducing_MainThm1_to_two_others}
\begin{proof}
We will derive Theorem~\ref{MainThm1}
from Theorems~\ref{MainHomotopy}, \ref{th:psv}, \ref{MainThmGeometric}
which are proved already.

If the regular semi-local domain contains an infinite field, then this derivation is done
in
\cite[Section 3]{FP}.
So, it remains to do such a reduction in the case when the regular semi-local domain contains a finite field.
Since the arguments given in \cite[Section 3]{FP} are detailed,
our arguments will be sketchy.

Theorems~\ref{th:psv} and ~\ref{MainHomotopy} yield Theorem~\ref{MainThm1} in the simple simply-connected case,
when the ring $R$ is the semi-local ring of finitely many {\it closed} points on a $k$-smooth variety (the field $k$ is finite).
Indeed, let $k$ be a field and let $\mathcal O$, $\bG$ and $\mathcal G$
be as in Theorem~\ref{MainHomotopy}.
Assume additionally that the group scheme $\bG$ is simple and simply-connected.
Take $\mathcal G_t$ from the conclusion of Theorem~\ref{MainHomotopy}  for the bundle $\mathcal E_t$ from Theorem~\ref{th:psv}.
Then take $h(t)$ from the conclusion of Theorem~\ref{MainHomotopy} for the bundle $\mathcal E_t$ from Theorem\ref{th:psv}.
Finally take the inclusion $i_0: U \hookrightarrow U\times \mathbb A^1$ for the section $s$
from the Theorem~\ref{th:psv}.
Theorem~\ref{th:psv}  with this choice of $\mathcal E_t$, $h(t)$ and $s$
yields triviality of the $\bG$ bundle $\mathcal G$ from Theorem~\ref{MainHomotopy}.

{\it The nearest aim is}
to show that
Theorem~\ref{MainThm1}
holds for the ring $R$ as above in this section and
{\it for arbitrary
semi-simple simply-connected group scheme $\bG$ over $R$
}.
By~\cite[Exp. XXIV 5.3, Prop. 5.10]{SGA3} the category of semi-simple simply connected
group schemes over a Noetherian domain $R$ is semi-simple. In
other words, each object has a unique decomposition into a product
of indecomposable objects. Indecomposable objects can be described
as follows. Take a domain $R^{\prime}$ such that $R\subseteq
R^{\prime}$ is a finite \'{e}tale extension and a simple
simply connected group scheme $\bG^{\prime}$ over $R^{\prime}$. Now,
applying the Weil restriction functor $\R_{R^{\prime}/R}$ to the
$R$-group scheme $\bG^{\prime}$ we get a simply connected $R$-group
scheme $\R_{R^{\prime}/R}(\bG^{\prime})$, which is
an indecomposable object in the above category. Conversely, each
indecomposable object can be constructed in this way.

Take a decomposition of $\bG$ into indecomposable factors
$\bG=\bG_1\times \bG_2\times\dots\times \bG_r$. Clearly, it suffices to
check that for each index $i$ the kernel of the map
$$ H^1_{\text{\'et}}(R,\bG_i)\to H^1_{\text{\'et}}(K,\bG_i) $$
\noindent
is trivial. We know that there exists a finite \'{e}tale extension
$R^{\prime}_i/R$ such that $R^{\prime}_i$ is a domain and the Weil
restriction $\R_{R^{\prime}_i/R}(\bG^{\prime}_i)$ coincides with
$\bG_i$.

The Faddeev---Shapiro Lemma \cite[Exp. XXIV Prop. 8.4]{SGA3} states
that there is a canonical isomorphism
$$
H^1_{\text{\'et}}\big(R, \R_{R^{\prime}_i/R}(\bG^{\prime}_i)\big)
\cong H^1_{\text{\'et}}\big(R^{\prime}_i,\bG^{\prime}_i\big)
$$
\noindent
that preserves the distinguished point.
Applying Theorem
\ref{MainThm1} to the semi-local regular domains $R^{\prime}_i$, its
fraction fields $K_i$, and the $R^{\prime}_i$-group schemes
$\bG^{\prime}_i$ we see that
Theorem~\ref{MainThm1}
holds for the ring $R$ as above in this section and
{\it for arbitrary
semi-simple simply-connected
}
group scheme $\bG$ over $R$.

Now Theorem \ref{MainThmGeometric}
yields Theorem~\ref{MainThm1}
for the ring $R$ as above in this section and for arbitrary
reductive group scheme $\bG$ over $R$.

This latter case
implies easily Theorem~\ref{MainThm1} for arbitrary reductive group scheme over a ring $R$,
where $R$ is the semi-local ring of finitely many {\it arbitrary} points on a $k$-smooth irreducible affine $k$-variety
(see \cite[Lemma 3.3]{FP}). Arguments using Popescu's Theorem complete the proof of Theorem~\ref{MainThm1}
(see \cite[Proof of Theorem 1]{FP}). Those arguments runs now easier since Theorem~\ref{MainThm1} is established already
for semi-local rings of finitely many {\it arbitrary} points on a $k$-smooth variety (with a finite field $k$).
{\it The proof of the theorem~\ref{MainThm1}
is completed.
}

\end{proof}

\section{Two exact sequences}
\label{Section_Aus_Buksbaum_3}
Let $k$ be a field.
\begin{thm}
\label{Aus_Buksbaum_3}
Let $\mathcal O$ be the semi-local ring of finitely many closed points
on a $k$-smooth irreducible affine $k$-variety $X$.
Let $K=k(X)$.
Let
$$\mu: \bG \to \mathbf C$$
be a smooth $\mathcal O$-morphism of reductive
$\mathcal O$-group schemes, with a torus $\mathbf C$.
Suppose additionally that
the kernel of $\mu$ is a reductive $\mathcal O$-group scheme.
Then the following sequence
\begin{equation}
\label{Aus_Buks_sequence_3}
\{0\} \to \mathbf C(\mathcal O)/\mu(\bG(\mathcal O)) \to
\mathbf C(K)/\mu(\bG(K)) \xrightarrow{\sum res_{\mathfrak p}} \bigoplus_{\mathfrak p}
\mathbf C(K)/[\mathbf C(\mathcal O_{\mathfrak p})\cdot \mu(\bG(K))] \to \{0\}
\end{equation}
is exact, where $\mathfrak p$ runs over
the height $1$ primes of $\mathcal O$ and $res_{\mathfrak p}$ is the natural map (the projection to the factor group).
\end{thm}

\begin{proof}
In fact, the exactness at the left hand side  term is a direct consequence of
Theorem \ref{MainThm1}. The exactness at the right hand side term is
proved in
\cite{C-T-S}.
Theorem \cite[Thm.1.1]{Pan2} yields exactness of the sequence
(\ref{Aus_Buks_sequence_3})
at the middle term.
\end{proof}

The last aim is to prove theorem \ref{PurityForSubgroup}.
Let $k$, $\mathcal O$ and $K$ be as in Theorem
\ref{Aus_Buksbaum_3}.
Let $G$ be a semi-simple $\mathcal O$-group scheme.
Let
$i: Z \hra G$ be a closed subgroup scheme of the center $Cent(G)$.
{\bf It is known that $Z$ is of multiplicative type}.
Let $G'=G/Z$ be the factor group,
$\pi: G \to G'$ be the projection.
It is known that $\pi$ is finite surjective and faithfully flat. Thus
the sequence of $\mathcal O$-group schemes
\begin{equation}
\label{ZandGndGprime}
\{1\} \to Z \xra{i} G \xra{\pi} G^{\prime} \to \{1\}
\end{equation}
induces an exact sequence of group sheaves in $\text{fppt}$-topology.
Thus for every $\mathcal O$-algebra $R$ the sequence
(\ref{ZandGndGprime})
gives rise to a boundary operator
\begin{equation}
\label{boundary}
\delta_{\pi,R}: G'(R) \to \textrm{H}^1_{\text{fppf}}(R,Z)
\end{equation}
One can check that it is a group homomorphism
(compare \cite[Ch.II, \S 5.6, Cor.2]{Se}).
Set
\begin{equation}
\label{AnotherFunctor}
{\cal F}(R)= \textrm{H}^1_{\text{fttf}}(R,Z)/ Im(\delta_{\pi,R}).
\end{equation}
Clearly we get a functor on the category of $\mathcal O$-algebras.
\begin{thm}
\label{PurityForSubgroup}
Let $\mathcal O$ be the semi-local ring of finitely many {\bf closed} points
on a $k$-smooth irreducible {\bf affine} $k$-variety $X$.
Let $G$ be a semi-simple $\mathcal O$-group scheme.
Let
$i: Z \hra G$ be a closed subgroup scheme of the center $Cent(G)$.
Let
${\cal F}$
be the functor
on the category
$\mathcal O$-algebras
given by
(\ref{AnotherFunctor}).
Then the sequence
\begin{equation}
\label{Aus_Buks_sequence_2}
0\to {\cal F}(\mathcal O) \to
{\cal F}(K) \xrightarrow{\sum can_{\mathfrak p}} \bigoplus_{\mathfrak p}
{\cal F}(K)/Im[{\cal F}(\mathcal O_{\mathfrak p}) \to {\cal F}(K)] \to 0
\end{equation}
is exact, where $\mathfrak p$ runs over
the height $1$ primes of $\mathcal O$ and $can_{\mathfrak p}$ is the natural map (the projection to the factor group).
\end{thm}


\begin{proof}[Proof of Theorem \ref{PurityForSubgroup}]
The group $Z$ is of  multiplicative type.
So we can find a finite \'{e}tale $\mathcal O$-algebra $A$ and
a closed embedding
$Z \hra R_{A/{\mathcal O}}(\mathbb G_{m,\ A})$
into the permutation torus
$T^{+}=R_{A/\mathcal O}(\mathbb G_{m,\ A})$.
Let
$G^{+}=(G \times T^{+})/Z$
and
$T=T^{+}/Z$,
where
$Z$ is embedded in
$G \times T^{+}$
diagonally.
Clearly
$G^{+}/G=T$.
Consider a commutative diagram
$$
\xymatrix {
{}           & \{1\}                       & \{1\} \\
{}           & {G'} \ar[r]^{id}  \ar[u]          & {G'}   \ar[u]          \\
\{1\} \ar[r] & {G} \ar[r]^{j^+} \ar[u]^{\pi} & {G^+} \ar[r]^{\mu^+} \ar[u]^{\pi^+} & {T} \ar[r] & \{1\} \\
\{1\} \ar[r] & {Z} \ar[r]^{j} \ar[u]^{i} & {T^+} \ar[r]^{\mu} \ar[u]^{i^+} & {T} \ar[u]_{id} \ar[r] &  \{1\} \\
{}           & \{1\} \ar[u]                      & \{1\} \ar[u]\\
}
$$
with exact rows and columns.
By the known fact (see Lemma
\ref{FlatAndEtale})
and Hilbert 90
for the semi-local $\mathcal O$-algebra $A$ one has
$\textrm{H}^1_{\text{fppf}}(\mathcal O,T^+)= \textrm{H}^1_{\text{\'et}}(\mathcal O,T^+)= \textrm{H}^1_{\text{\'et}}(A,\Bbb G_{m,A})=\{* \}$.
So, the latter diagram gives rise to
a commutative diagram of pointed sets
$$
\xymatrix {
{}           & {}                            & {\textrm{H}^1_{\text{fppt}}(\mathcal O,G')} \ar[r]^{id}            & {\textrm{H}^1_{\text{fppt}}(\mathcal O,G')}            \\
{G^+(\mathcal O)} \ar[r]^{\mu^+_\mathcal O} & {T(\mathcal O)} \ar[r]^{\delta^+_\mathcal O}  & {\textrm{H}^1_{\text{fppt}}(\mathcal O,G)} \ar[r]^{j^+_*} \ar[u]^{\pi_*} & {\textrm{H}^1_{\text{fppt}}(\mathcal O,G^+)} \ar[u]^{\pi^+_*}  \\
{T^+(\mathcal O)} \ar[r]^{\mu_\mathcal O} \ar[u]^{i^+_*} & {T(\mathcal O)} \ar[r]^{\delta_\mathcal O} \ar[u]^{id} & {\textrm{H}^1_{\text{fppt}}(\mathcal O,Z)} \ar[r]^{\mu} \ar[u]^{i_*} & {\{*\}} \ar[u]^{i^+_*}  \\
{}           & {} & {G'(\mathcal O)} \ar[u]^{\delta_{\pi}}                      & {} \\
}
$$
with exact rows and columns. It follows that
$\pi^+_*$ has  trivial kernel and one has a chain of group isomorphisms
\begin{equation}
\label{F=T/mu_+(G_+)}
\textrm{H}^1_{\text{fppf}}(\mathcal O,Z) / Im (\delta_{\pi,\mathcal O})= ker (\pi_*)= ker (j^+_*) = T(\mathcal O)/\mu^+ (G^+(\mathcal O)).
\end{equation}
Clearly these isomorphisms respect $\mathcal O$-homomorphisms of semi-local $\mathcal O$-algebras.

The morphism $\mu^{+}: G^{+} \to T$ is a smooth $\mathcal O$-morphism of reductive
$\mathcal O$-group schemes, with the torus $T$. The kernel
$ker(\mu^{+})$
is equal to $G$ and $G$ is a reductive $\mathcal O$-group scheme.

The isomorphisms (\ref{F=T/mu_+(G_+)}) show that the sequence
(\ref{Aus_Buks_sequence_2})
is isomorphic to the sequence
(\ref{Aus_Buks_sequence_3}).
The latter sequence is exact by Theorem
\ref{Aus_Buksbaum_3}.
Whence the theorem.
\end{proof}

\end{document}